\documentclass[leqno,12pt]{article}
\usepackage{inputenc}
\usepackage[english]{babel}
\usepackage{amsmath}
\usepackage{amssymb}
\usepackage{amsthm}
\usepackage[all]{xy}
\usepackage[affil-it]{authblk}

\newtheorem{thm}{Theorem}[section]
\newtheorem{lem}[thm]{Lemma}
\newtheorem{prop}[thm]{Proposition}
\newtheorem{cor}[thm]{Corollary}

\theoremstyle{definition}
\newtheorem{ex}[thm]{\it Example}

\newtheorem{rem}[thm]{\it Remark}

\numberwithin{equation}{section}

\usepackage{color}  

\title{One half of almost symmetric numerical semigroups}
\date{}

\begin{document}
\newcommand{\Hi}{\mathrm {Hilb}}
\newcommand{\card}{\mathrm {card}}
\newcommand{\ap}{\mathrm {Ap}}
\newcommand{\ord}{\mathrm {ord}}
\newcommand{\mapM}{\mathrm {maxAp_M}}
\newcommand{\map}{\mathrm {maxAp}}
\newcommand{\gr}{\mathrm{gr}}
\newcommand{\al}{\boldsymbol{\alpha}}
\newcommand{\be}{\boldsymbol{\beta}}
\newcommand{\ga}{\boldsymbol{\gamma}}
\newcommand{\0}{\boldsymbol 0}
\newcommand{\de}{\Delta^S}
\newcommand{\du}{S \! \Join^b \! E}
\author{
F. Strazzanti \thanks{{\em e-mail}: strazzanti@mail.dm.unipi.it}}

\affil{Dipartimento di Matematica, Universit\`{a} degli Studi di Pisa,
Largo Bruno Pontecorvo 5, 56127 Pisa, Italy}

\maketitle

\begin{abstract}

Let $S,T$ be two numerical semigroups. We study when $S$ is one
half of $T$, with $T$ almost symmetric. If we assume that the type of $T$,
$t(T)$, is odd, then for any $S$ there exist infinitely many
such $T$ and we prove that $1 \leq t(T) \leq 2t(S)+1$. On the
other hand, if $t(T)$ is even, there exists such $T$ if and only
if $S$ is almost symmetric and different from $\mathbb{N}$; in
this case the type of $S$ is the number of even pseudo-Frobenius
numbers of $T$. Moreover, we construct these families of
semigroups using the numerical duplication with respect to a
relative ideal.

\medskip


\\ Almost symmetric semigroup $\cdot$ Numerical duplication
$\cdot$ type.
\end{abstract}

\section{Introduction}

A numerical semigroup $S$ is a submonoid of $\mathbb{N}$ such that
$\mathbb{N} \setminus S$ is finite. Numerical semigroups arise in
several contexts, such as commutative algebra, algebraic geometry,
coding theory, number theory and combinatorics. Many classes of
numerical semigroups are defined translating ring concepts (see
e.g. \cite{BDF}); for example symmetric and pseudo-symmetric
numerical semigroups are the corresponding concepts of Gorenstein
and Kunz rings in numerical semigroup theory. Almost symmetric
numerical semigroups, that are the object of this paper, were
introduced in \cite{BF}, together with the corresponding notion of
almost Gorenstein rings, as generalization of symmetric and
pseudo-symmetric numerical semigroups; in fact these classes are
exactly the almost symmetric numerical semigroups of type $1$ and
$2$, respectively.

In \cite{RGSGU} Rosales, Garc\'ia-S\'anchez, Garc\'ia-Garc\'ia,
and Urbano-Blanco introduced the concept of one half of a
numerical semigroup in order to solve proportionally modular
diophantine inequalities; $S$ is one half of $T$ if $S=\{s \in
\mathbb{N} | \ 2s \in T \}$. In the last ten years several authors
have studied this concept and its generalizations, see for example
\cite{Do, MO, M, Sm} and the papers quoted below.

Rosales and Garc\'ia-S\'anchez proved in \cite{RGS2} that every
numerical semigroup is one half of infinitely many symmetric
semigroups and Swanson generalized this result in \cite{Sw}.
Moreover Rosales proved in \cite{R} that a numerical semigroup
(different from $\mathbb{N}$) is one half of a pseudo-symmetric
semigroup if and only if is symmetric or pseudo-symmetric.

In this paper we generalize these results to the case of almost
symmetric semigroups. According to results of \cite{R},
\cite{RGS}, and \cite{RGS2} we consider separately the cases of
almost symmetric semigroups with even and odd type.

Starting with a numerical semigroup $S$ with type $t$, in
\cite{DS} are constructed infinitely many almost symmetric
semigroups with type $1,3,5,7, \dots ,2t+1$, such that $S$ is
their half. This construction is called numerical duplication with
respect to a proper ideal of $S$ and arises in commutative
algebra, in fact it is the value semigroup of particular algebroid
branches (see \cite[Theorems 3.4 and 3.6]{BDS}). In this paper we
prove that if $S$ is one half of almost symmetric semigroup $T$
with odd type, then the type of $T$ is included in the values
above and all such semigroups can be constructed with the
numerical duplication with respect to a \textit{relative} ideal.

On the other hand if $T$ is almost symmetric with even type, $S$
is almost symmetric and its type is the number of even
pseudo-Frobenius numbers of $T$; in particular $t(S) \leq t(T)$.
Moreover we prove that a numerical semigroup different from
$\mathbb{N}$ is almost symmetric if and only if it is one half of
an almost symmetric semigroup with even type or equivalently of a
finite number of almost symmetric semigroups with even type.
Finally, we characterize these semigroups.

\vspace{1em}

The paper is organized as follows. In the Section $2$ we recall
some definitions and results about numerical semigroups and prove
some useful lemmas. In Section $3$ we introduce the numerical
duplication, prove that every numerical semigroups can be realized
as numerical duplication with respect to a relative ideal (see
Proposition \ref{duplication}), and we use this fact in Theorem
\ref{main odd} to characterize one half of almost symmetric
numerical semigroup $T$ with odd type; moreover, in Theorem
\ref{odd} we give bounds for the type of $T$ (see also the
discussion after the theorem). Finally, in the last section we
characterize when $T$ is almost symmetric with even type in terms
of properties of $\frac{T}{2}$ (see Theorem \ref{main even}) and
prove in Corollary \ref{final} that a numerical semigroup $S \neq
\mathbb{N}$ is almost symmetric if and only if it is one half an
almost symmetric numerical semigroup with even type.

\section{Preliminaries}

Let $S$ be a numerical semigroup. The maximum of $\mathbb{N}
\setminus S$ is called \textit{Frobenius number} of $S$ and we
denote it by $f(S)$. Clearly if $s \in S \setminus \{0\}$ then
$s+f(S) \in S$; more generally we define the set of
\textit{pseudo-Frobenius numbers} PF$(S)=\{x \in \mathbb{Z}
\setminus S | \ x+s \in S \text{ \ for any \ } s \in S \setminus
\{0\} \}$. The cardinality of PF$(S)$ is called the \textit{type}
of $S$; this name is due to ring theory (see e.g. the first
section of \cite[Chapter II]{BDF}).

Let $s$ be an integer such that $s \notin S$. If $f(S)-s \in S, s$
is called a gap of first type, otherwise a gap of second type. We
denote the set of gaps of second type with ${\rm L}(S)$; it is
easy to see that ${\rm PF}(S) \subseteq {\rm L}(S) \cup \{f(S)\}$.

If we have that $s \in S$ if and only if $f(S)-s \notin S$, then
we say that $S$ is \textit{symmetric}; if $f(S)$ is even and this
property holds for any $s \in \mathbb{Z}$ but $f(S)/2$, we call
$S$ \textit{pseudo-symmetric}. Clearly these properties mean that
${\rm L}(S)=\emptyset$ and ${\rm L}(S)=\{\frac{f(S)}{2}\}$ respectively. 
Finally if ${\rm L}(S) \subseteq {\rm PF}(S)$, 
we call $S$ \textit{almost symmetric}. It is well
known that $S$ is symmetric if and only if has type $1$, while the
pseudo-symmetric numerical semigroups have type $2$ (but there are
numerical semigroups with type $2$ that are not pseudo-symmetric).
Almost symmetric semigroups generalize these two classes, in
particular symmetric and pseudo-symmetric numerical semigroups are
exactly almost symmetric semigroups with type $1$ and $2$
respectively (see \cite[Proposition 7]{BF}).

A \textit{relative ideal} of $S$ is a set $E \subseteq \mathbb{Z}$
such that $E+S \subseteq E$ and $x+E \subseteq S$ for some $x \in
S$; moreover if $E \subseteq S$, we say simply that $E$ is a
\textit{(proper) ideal} of $S$. We denote with $f(E)$ the Frobenius number
of $E$, i.e. the maximum of $\mathbb{Z} \setminus E$. For example
$M(S)=S \setminus \{0\}$ and $K(S)=\{x \in \mathbb{Z}| \ f(S)-x
\notin S\}$ are relative ideals of $S$ (the first one is a proper
ideal) and are called \textit{maximal ideal} and \textit{standard
canonical ideal}, respectively. More generally we say that $E$ is
a \textit{canonical ideal} of $S$ if $E=K(S)+x$ for some $x \in
\mathbb{Z}$. The names of these two ideals come from ring theory 
and they are very important, for example it
is known that $S$ is symmetric if and only if $S=K(S)$ and it is
almost symmetric if and only if $M(S)+K(S) \subseteq M(S)$ (see
\cite[Proposition 4]{BF}). We note the analogy with ring theory,
where a Cohen Macaulay local ring $(R, \mathfrak{m})$ is Gorenstein
if and only if is isomorphic to its canonical module $K$, and is 
almost Gorenstein when $\mathfrak{m}+K \subseteq \mathfrak{m}$,
where $R \subseteq K \subseteq \overline{R}$.

If $E$ and $F$ are relative ideals of $S$, we define $E-F=\{x \in
\mathbb{Z}| \ x+f \in E \text{ \ for any \ } f \in F\}$, that is
also a relative ideal. For example we can define $M(S)-M(S)$ in
this way and it is easy to check that this is a numerical semigroup
satisfying the equality $M(S)-M(S)=S \cup \text{PF}(S)$.

There exist other characterizations of almost symmetric numerical
semigroups; in the following sections we will use the next two.

\begin{lem} \label{almost symmetric}
Let $S$ be a numerical semigroup with Frobenius number $f$. Then
$S$ is almost symmetric if and only if the following property
holds,
$$
s \in S \Longleftrightarrow f-s \notin S \cup {\rm PF}(S)
$$
for any $s \in \mathbb{Z} \setminus \{0\}$.
\end{lem}

\begin{proof}
First of all let $s$ be an element of $S$. One has $f-s \notin S
\cup \text{PF}(S)$ or otherwise $f=s+(f-s) \in S$, a
contradiction. This is always true.

Now let $s \notin S$. The condition above is equivalent to $f-s
\in S \cup \text{PF}(S)$. Clearly, the set of elements such that
$f-s \notin S$ for some $s \notin S$ is L$(S)$, then this
condition is equivalent to L$(S) \subseteq \text{PF}(S)$, that is $S$
is almost symmetric.
\end{proof}

The next theorem was proved by Nari in \cite{N}, Theorem $2.4$.

\begin{thm} \label{Nari}
Let $S$ be a numerical semigroup and {\rm PF}$(S)=\{ f_1 < \dots < f_{t-1} < f \}$,
then the following conditions are equivalent: \\
{\rm (1)} $S$ is almost symmetric; \\
{\rm (2)} $f_i+f_{t-i}=f$ for any $i \in \{1, \dots,t-1\}$.
\end{thm}

In the rest of the paper we distinguish between almost symmetric
numerical semigroups with odd and even type. Luckily from \cite[Theorem 3]{RG2}
can be deduced a nice distinction between them; here we give a direct proof of this fact.

\begin{prop} \label{frobenius odd}
Let $T$ be an almost symmetric numerical semigroup, then $T$ has odd type if and only if $f(T)$ is odd.
\end{prop}

\begin{proof}
Let PF$(S)=\{f_1 < \dots <f_{t-1} < f\}$. By Theorem \ref{Nari}
one has $f_i+f_{t-i}=f$, then $t$ is even if and only if $f/2 \in
{\rm PF}(S)$. Consequently if $f$ is odd then $t$ is odd.
Conversely suppose that $f$ is even; note that $f/2 \notin S$,
otherwise the Frobenius number is in $S$, and that, by Lemma
\ref{almost symmetric}, $f/2=f-f/2 \in S \cup {\rm PF}(S)$; hence
$f/2 \in {\rm PF}(S)$.
\end{proof}

Finally we remember that a numerical semigroup $S$ is one half of
$T$ if $S = \{ s \in \mathbb{N} | \ 2s \in T \}$ and in this case
we will write $S= \frac{T}{2}$.

\section{One half of almost symmetric numerical semigroups with odd type}

In this section we study those numerical semigroups that are one
half of almost symmetric semigroups with odd type.

\begin{thm} \label{odd}
Let $T$ be an almost symmetric numerical semigroups with odd type
$t$. If $S$ is one half of $T$ then $t(S)\geq (t-1)/2$.
\end{thm}

\begin{proof}
Let PF$(T)=\{ f_1 < \dots < f_{t-1} < f \}$; by Theorem \ref{Nari}
one has $f_i+f_{t-i}=f$. In particular, since $f$ is odd, in ${\rm
PF}(T)$ there are $(t-1)/2$ even elements and $1+(t-1)/2$ odd
elements.

Consider an even pseudo-Frobenius number $f_i=2e_i$; clearly $e_i
\notin S$ and we claim that $e_i \in {\rm PF}(S)$. Let $s \in S
\setminus \{0\}$, then $2s \in T$ and therefore $2(e_i+s)=f_i+2s
\in T$, because $f_i \in {\rm PF}(T)$; consequently $e_i+s \in S$
for any $s \in S \setminus \{0\}$, that is $e_i \in {\rm PF}(S)$.

Hence there are at least $(t-1)/2$ pseudo-Frobenius numbers in $S$.
\end{proof}

\begin{rem}
In general, there is not an upper bound for $t(S)$. In fact, in
\cite{RGS} is proved that every numerical semigroup is one half of
a symmetric numerical semigroup; then, even if we restrict to the case $T$
symmetric, $S$ may be {\it any} numerical semigroup.
\end{rem}

Let $S$ be a numerical semigroup, $E$ a proper ideal of $S$ and
$b$ an odd element of $S$. In \cite{DS} is defined the
\textit{numerical duplication} of $S$ with respect to $E$
as the numerical semigroup
$$
\du = 2\cdot S \cup (2\cdot E +b)
$$
where $2 \cdot S=\{2s | \ s \in S\}$ and $2 \cdot E=\{2e| \ e \in E\}$.

\vspace{1em}
This construction is motivated by a commutative
algebra construction (see \cite{BDS}), but we are interested in it
because can be used to construct almost symmetric semigroups. For
example Corollary $4.9$ of \cite{DS} shows that, starting with a
numerical semigroup $S$, it is possible to choose proper ideals $E_0,
E_1, \dots , E_{t(S)}$, such that $\du_0, \dots , \du_{t(S)}$ are
almost symmetric of type $1, 3, 5, \dots, 2t(S)+1$, respectively.

Coming back to Theorem \ref{odd}, the inequality of the statement
is equivalent to $t \leq 2t(S)+1$, so the previous remark shows
that this estimation is sharp.

Unfortunately there exist almost symmetric numerical semigroups
with odd type that cannot be constructed in this way. For example
consider $T=\langle 9,10,14,15 \rangle =
\{0,9,10,14,15,18,19,20,23,24,25,27,28,29,30,32 \rightarrow \}$,
where $\rightarrow$ means that all integers greater than $32$ are in $T$;
in this case $S=\frac{T}{2}=\{0,5,7,9,10,12,14 \rightarrow \}$.
The point is to choose $b \in S$. We must have $2 \cdot
E+b=\{9,15,19,23,25,27,29,33,35,37 \dots\}$, so we have
$$
\begin{array}{ll}
 E=\{2,5,7,9,10,11,12,14 \rightarrow \} \ \ \ \ \ \ \ \ \ \ \ \ \ \ \ \ \ \ \ \ \ \ \ & \text{if } b=5 \\
 E=\{1,4,6,8,9,10,11,13 \rightarrow \}  & \text{if } b=7 \\
 E=\{0,3,5,7,8,9,10,12 \rightarrow \}  & \text{if } b=9 \\
 E \text{ contains a negative element} & \text{if } b>9
\end{array}
$$

in any case $E$ is not contained in $S$ and then $E$ is not a proper ideal of $S$.

\vspace{1em}

To solve this problem note that, if $E$ is not a proper ideal, but
a relative ideal such that $b+E+E \subseteq S$, then $\du$ is
still a numerical semigroup. Under this easy generalization, we
can construct all numerical semigroups.

\begin{prop} \label{duplication}
Every numerical semigroup $T$ can be realized as numerical
duplication $\du$, where $S=\frac{T}{2}$, $b$ is an odd element of
$S$ and $E$ is a relative ideal of $S$ such that $b+E+E \subseteq
S$.
\end{prop}

\begin{proof}
Let $b$ be an odd element of $S$ and set $E=\frac{T-b}{2}$
(it is possible that $T-b$ contains negative elements and in 
this case we have negative elements in $\frac{T-b}{2}$).
Suppose that $E$ is not a relative ideal of $S$, that is, there
exist $s \in S$ and $e \in E$ such that $s+e \notin E$;
consequently $2(s+e)+b \notin T$, but $2s+(2e+b) \in T+T \subseteq
T$, since $s \in S$ and $e \in E$; contradiction. This means that
$E$ is a relative ideal of $S$.

Let $e,e'$ be two elements of $E$, then $b+2e$ and $b+2e'$ is in
$T$; therefore $2b+2e+2e' \in T$ and it is equivalent to $b+e+e'
\in S$. Hence $b+E+E \subseteq S$.

Finally, by construction, it is clear that $T=\du$ .
\end{proof}

Note that in the previous proof we have not determined $b$, so
there exist infinitely many ways to obtain the semigroup $T$ as a
numerical duplication.

The following corollary is straightforward.

\begin{cor} \label{fraction}
Let $S$ be a numerical semigroup. Then every semigroup $T$ such
that $S=\frac{T}{2}$, is equal to $\du$ for some relative ideal $E$
and an odd integer $b \in S$.
\end{cor}

Theorem 4.3 of \cite{DS} characterizes almost symmetric semigroups
realized as numerical duplication with respect to a proper ideal
$E$. We will see that in our case only one implication is true.

Let $\widetilde E$ denote the relative ideal $E-e$, where $e
:=f(E)-f(S)$; clearly $f(\widetilde E)=f(S)$. Moreover, we set
$f:=f(S), M:=M(S)$ and $K:=K(S)$.

First of all we note that, by definition, the Frobenius number of
$T=\du$ is the maximum between $2f(S)$ and $2f(E)+b$. If $E$ is
proper, then $f(E) \geq f(S)$ and so $f(T)=2f(E)+b$; in general it
is possible that $2f(S)>2f(E)+b$: in fact, by Proposition
\ref{duplication}, $f(T)$ can be even. However, if $T$ is an
almost symmetric numerical semigroup with odd type,  we have
$f(T)=2f(E)+b$ by Proposition \ref{frobenius odd}. Thanks to this,
the proof of one implication of \cite[Theorem 4.3]{DS} works also
if $E$ is a relative and not proper ideal of $S$ (obviously
provided $E+E+b \subseteq S$); hence we will not write the proof
of the following proposition.

\begin{prop} \label{first part}
Let $T$ be an almost symmetric numerical semigroup with odd type.
Then there exist a numerical semigroup $S$, a relative ideal $E$
of $S$ and an odd integer $b \in S$ such that $T=\du$. Moreover
for every choose of such $S,E,b$ one has $f(T)=2f(E)+b, K-(M-M)
\subseteq \widetilde E \subseteq K$, and $K-\widetilde E$ is a
numerical semigroup.
\end{prop}

The converse of the previous result is not true. Consider the
numerical semigroup $S=\{0,4,5,6,8 \rightarrow\}$ and the relative
ideal $E=\{2,3,4,6 \rightarrow \}$. It is straightforward to check
that $K-(M-M)=M=\widetilde E, K=S, E+E+5 \subseteq S, K- \widetilde
E=M-M$ and $2f(E)+5 > 2f(S)$; then $T= S \! \Join^5 \! E$ is a
numerical semigroup, $K-(M-M) \subseteq \widetilde E \subseteq K$
and $K-\widetilde E$ is a numerical semigroup. However
$T=\{0,8,9,10,11,12,13,16 \rightarrow \}$ is not almost symmetric
because $1 \in {\rm L}(T) \setminus {\rm PF}(T)$.

If we look at the proof of \cite[Theorem 4.3]{DS}, we see that it
works also when $E$ is a relative ideal except for case {\bf
(iii)}. From this observation it comes out the idea of the next
theorem.

First of all, we recall and generalize some results of \cite{DS}.
The standard canonical ideal of $\du$ is the set of elements
$$ z=f(\du)-a \ \
\text{with} \ \
  \begin{cases}
    \frac{a}{2} \notin S,  & a \ \text{even}, \\
    \frac{a-b}{2} \notin E, & a \ \text{odd}.
  \end{cases}
$$

The next lemma is proved in \cite[Lemma 4.1 and Lemma 4.2]{DS}; in
the original statement $E$ is a proper ideal, but the proof works
also when $E$ is a relative ideal.

\begin{lem}\label{2}
Let $E$ be a relative ideal of $S$. Assume that $K-(M-M) \subseteq \widetilde E$. Then we have: \\
{\rm (1)} for any $x \notin E$, $f(E)-x \in M-M$; \\
{\rm (2)} if moreover $K-\widetilde E$ is a numerical semigroup, then, for any $x \notin
E$, $f(E)-x \in E-E$.
\end{lem}

Now we can construct every almost symmetric numerical semigroup
with odd type.

\begin{thm} \label{main odd}
A numerical semigroup $T$ is almost symmetric with odd type
if and only if there exist a relative ideal $E$ of $S:=\frac{T}{2}$
and an odd integer $b \in S$ such that:  \\
{\rm (1)} $T= \du$; \\
{\rm (2)} $f(T)=2f(E)+b$; \\
{\rm (3)} $K-(M-M) \subseteq \widetilde E \subseteq K$; \\
{\rm (4)} $K-\widetilde E$ is a numerical semigroup; \\
{\rm (5)} $b+e+E+K \subseteq M$.
\end{thm}

\begin{proof}
As we said above, the proof is only a modification of the proof of
\cite[Theorem 4.3]{DS}. However we include the complete proof for
the sake of completeness.

Assume that the five conditions of the statement hold and
prove that $T$ is almost symmetric, i.e. $M(T)+K(T) \subseteq M(T)$. There are four cases: \\\\
{\bf (i)} $2s \in M(T)$ and $2f(E)+b-a \in K(T)$, where $s \in
  M$, $a$ is even and $\frac{a}{2} \notin S$; \\
{\bf (ii)} $2s \in M(T)$ and $2f(E)+b-a \in K(T)$, where $s \in
  M$, $a$ is odd and $\frac{a-b}{2} \notin E$; \\
{\bf (iii)} $2t+b \in M(T)$ and $2f(E)+b-a \in K(T)$, where $t \in
  E$, $a$ is even and $\frac{a}{2} \notin S$; \\
{\bf (iv)} $2t+b \in M(T)$ and $2f(E)+b-a \in K(T)$, where $t \in
  E$, $a$ is odd and $\frac{a-b}{2} \notin E$. \\\\
{\bf (i)} Since \ $2s+2f(E)+b-a$ \ is odd, it belongs to $M(T)$ if
and only if \ $s+f(E)-\frac{a}{2} \in E$, \ i.e.  $s+f-\frac{a}{2}
\in \widetilde E$. Since $\frac{a}{2} \notin S$, i.e.
$f-\frac{a}{2} \in K$, we obtain $s+f-\frac{a}{2}\in M+K \subseteq
K-(M-M) \subseteq \widetilde E$.

\noindent {\bf (ii)} Since $2s+2f(E)+b-a$ is even, it belongs to
$M(T)$ if and only if $s+f(E)-\frac{a-b}{2} \in M$. Since
$\frac{a-b}{2}\notin E$, we can apply Lemma \ref{2} to obtain
$f(E)-\frac{a-b}{2} \in M-M$, that implies the thesis.

\noindent {\bf (iii)} Since $2t+b+2f(E)+b-a$ is even, it belongs
to $M(T)$ if and only if $t+b+f(E)-\frac{a}{2} \in M$, i.e. $t+b+e+f-\frac{a}{2} \in M$.
But this is true by Condition $5$, indeed $t \in E$ and $f-\frac{a}{2} \in K$.

\noindent {\bf (iv)} Since $2t+b+2f(E)+b-a$ is odd, it belongs to
$M(T)$ if and only if $t+f(E)-\frac{a-b}{2} \in E$. Since
$\frac{a-b}{2} \notin E$, the thesis follows immediately by Lemma
\ref{2}.

This proves that $T$ is almost symmetric; moreover, by Proposition
\ref{frobenius odd} its type is odd.

Conversely, we have already seen in Proposition \ref{first part}
that the first four conditions hold. Moreover, it is easy to see
that the last one is true: in fact we can use the same argument of
case {\bf (iii)} above.
\end{proof}

\begin{rem}
If the conditions of the previous theorem are satisfied, then
$E=e+\widetilde E \subseteq e+K$; consequently, thanks to the
fifth condition, we always have $E+E+b \subseteq S$.
\end{rem}

Anyway we notice that there are not problems when $E$ is a
canonical ideal, in fact, if we consider only semigroups with odd
Frobenius number, the proof of \cite[Proposition 3.1]{DS} still
works. So we have the following theorem:

\begin{thm}
The numerical semigroup $\du$ is symmetric if and only if $2f(E)+b>2f(S)$ and $E$ is a canonical ideal of $S$.
\end{thm}

\begin{cor}
Let $S$ be a numerical semigroup. Then the family of all symmetric
numerical semigroups $T$ such that $S=\frac{T}{2}$ is
$$
\mathcal D(S)=\{\du | \ E+E+b \subseteq S \rm{\ and \ } E \rm{ \ is \ a \ canonical \ ideal \ of \ } S \}
$$
\end{cor}

\begin{proof}
By Proposition \ref{duplication}, all semigroups can be
realized as numerical duplication with respect to a relative ideal 
$E$ such that $E+E+b \subseteq S$. Hence, recalling Corollary 
\ref{fraction}, we can use the previous theorem.

Finally note that if $E=K+x$, then $E+E+b \subseteq S$ implies
$2x+b>0$, because $0 \in K$. Since $f(E)=f(K)+x=f+x$, we have
$2f<2(f+x)+b=2f(E)+b$.
\end{proof}

Notice that $\mathcal D(S)$ is constructed by Rosales and
Garc\'ia-S\'anchez in \cite{RGS2} in a different way, but it is
easy to see that they coincide.

\section{One half of almost symmetric numerical semigroups with even type}

In this section we study when $T$ is almost symmetric with even
type or, equivalently, with even Frobenius number.

\begin{lem} \label{1}
Let $T$ be a numerical semigroup and ${\rm PF}(T)=\{f_1 < \dots < f_t \}$. Set $S:=\frac{T}{2}$. \\
{\rm (1)} If $f_i$ is even, then $\frac{f_i}{2} \in {\rm PF}(S)$.
In particular, the type of $S$ is greater than or equal to the number of even pseudo-Frobenius numbers of $T$. \\
{\rm (2)} If $f_t$ is even, then $f(S)=\frac{f_t}{2}$.
\end{lem}

\begin{proof}
{\rm (1)} If $f_i$ is even then $\frac{f_i}{2} \in \mathbb{Z}
\setminus S$, since $f_i \notin T$. Let $s$ be a positive element
of $S$, then $2s \in T$ and $f_i+2s \in T$,
since $f_i \in {\rm PF}(T)$; hence $\frac{f_i}{2}+s \in S$ and then $\frac{f_i}{2} \in {\rm PF}(S)$. \\
{\rm (2)} See \cite[Lemma 6.9]{RG}.
\end{proof}

In \cite{R} is proved that one half of a pseudo-symmetric
numerical semigroup is symmetric or pseudo-symmetric; we also know
that these classes consist of the almost symmetric semigroup with
type $1$ or $2$, respectively. In the next theorem this result is
generalized for any almost symmetric numerical semigroup with even
type.

\begin{thm}
If $T$ is almost symmetric with even Frobenius number, then
$S:=\frac{T}{2}$ is almost symmetric and its type is exactly the
number of even pseudo-Frobenius numbers of $T$.
\end{thm}

\begin{proof}
Let PF$(S)=\{f_1 < \dots < f_t\}$ and $i \in \{1, \dots, t\}$.
Since $f_i \notin S$, $2f_i \notin T$ and then, thanks to Lemma
\ref{almost symmetric} and to Lemma \ref{1},
$2(f_t-f_i)=2f_t-2f_i=f(T)-2f_i \in T \cup {\rm PF}(T)$.

If $2(f_t-f_i) \in T$, then $s=f_t-f_i \in S$ and therefore
$f_t=f_i+s$. If $s \neq 0$, then $f_t \in S$, since $f_i \in {\rm
PF}(S)$; hence $s=0$, that is $f_t=f_i$.

Consequently if $i \in \{ 1, \dots, t-1\}$, $2f_{t-1} = 2(f_t-f_i) \in 
{\rm PF}(T)$. In this way we obtain $t-1$ even pseudo-Frobenius number;
therefore, since $2f_t(S)$ is not included in this list, there are
at least $t$ even pseudo-Frobenius numbers in $T$, and, by the
previous lemma, they are exactly $t$.

Finally using Theorem \ref{Nari} it is straightforward to check that $S$ is almost symmetric.
\end{proof}

As in the previous section, starting with a numerical semigroup
$S$, we can construct all almost symmetric numerical semigroups
$T$ with even type, such that $S=\frac{T}{2}$; for this aim we use
numerical duplication again.

In \cite{DS} this is not possible, because the Frobenius number of
the numerical duplication with respect to a proper ideal is always
odd.

As in the previous section, $S$ will be a numerical semigroup and we set $f:=f(S), M:=M(S)$ and $K:=K(S)$.

Let us start with some lemmas. The first one was proved by
J\"ager; for the proof see \cite[Hilfssatz 5]{J}.

\begin{lem} \label{Jager}
For any relative ideal $E$, $K-E=\{x \in \mathbb Z \ | \ f-x \notin E\}$.
\end{lem}

\begin{lem}
Suppose that $S$ is almost symmetric and $E$ a relative ideal such that $E+E+b \subseteq S$. 
Assume that $2f \geq 2f(E)+b$. Then the following conditions are equivalent: \\
{\rm (1)} ${\rm PF}(S) \subseteq E-E$. \\
{\rm (2)} $M-M \subseteq E-E$. \\
{\rm (3)} $K \subseteq E-E$.
\end{lem}

\begin{proof}
First we claim that $f \in E-E$. Suppose that there exists $e \in E$ such that $f+e \notin E$. 
We have $2e+b>0$, since it is odd and $E+E+b \subseteq S$. Therefore $2f < 2(f+e)+b \leq 2f(E)+b$. 
Hence we get a contradiction, since $2f(E)+b < 2f$.

Now, by definition of almost symmetric semigroups, we have:
$$
M-M=S \cup {\rm PF}(S)=S \cup {\rm L}(S)\cup \{f\}=K \cup \{f\}.
$$
Moreover, $S \subseteq E-E$ because $E$ is an ideal, then, since $f \in E-E$, we obtain:
$$
{\rm PF}(S) \subseteq E-E \Longleftrightarrow S \cup {\rm PF}(S) \subseteq E-E \Longleftrightarrow
$$
$$
\Longleftrightarrow M-M \subseteq E-E \Longleftrightarrow K \cup \{f\} \subseteq E-E \Longleftrightarrow K \subseteq E-E.
$$
\end{proof}

\begin{lem} \label{K-E}
Let $S$ be almost symmetric and $E$ a relative ideal. If the equivalent conditions of the previous lemma hold, then $M-E = K-E$.
\end{lem}

\begin{proof}
Since $M \subseteq K$, one has $M-E \subseteq K-E$. Suppose that the equality does not hold,
i.e. there exists $x \in (K-E)\setminus (M-E)$; then there exists $e \in E$ such that $x+e \in K \setminus M$. Since $x+e \in K$, one has $f-x-e \notin S$ and then 
$f-x-e \in {\rm L}(S) \cup \{f\} = {\rm PF}(S) \subseteq E-E$ by assumptions. 
Hence $f-x=(f-x-e)+e \in E$ and, since $x \in K-E$, we obtain $f=(f-x)+x \in K$ that is a contradiction.
\end{proof}

\begin{thm} \label{main even}
Let $S$ be a numerical semigroup, let $b\in S$ be an odd integer, and let $E$ be a relative ideal of $S$
such that $E+E+b \subseteq S$ and $2f > 2f(E)+b$. 
Then the numerical duplication $T:=\du$ is almost symmetric (with even type) if and only if the following properties hold: \\
{\bf (i)}  $S$ is almost symmetric; \\
{\bf (ii)} $M-E \subseteq (E-M)+b$; \\
{\bf (iii)} $K \subseteq E-E$.
\end{thm}

\begin{proof}
$T$ is almost symmetric if and only if $M(T)+K(T) \subseteq K(T)$ and, recalling the characterization 
of $K(T)$, given before Lemma \ref{2}, this is equivalent to the following four conditions: \\\\
{\bf (i)} $2m+2f-a \in M(T)$ for any $m \in M$ and $a$ even such that $\frac{a}{2} \notin S$; \\
{\bf (ii)} $2m+2f-a \in M(T)$ for any $m \in M$ and $a$ odd such that $\frac{a-b}{2} \notin E$; \\
{\bf (iii)} $2e+b+2f-a \in M(T)$ for any $e \in E$ and $a$ even such that $\frac{a}{2} \notin S$; \\
{\bf (iv)} $2e+b+2f-a \in M(T)$ for any $e \in E$ and $a$ odd such that $\frac{a-b}{2} \notin E$. \\\\
Discussing every condition, we will see that {\bf (i), (ii), (iii)} are equivalent to the properties listed in the statement, while condition {\bf (iv)} is always true, if we assume {\bf (i)} and {\bf (iii)}. The thesis follows immediately from these facts. \\\\
{\bf (i)} We have $2m+2f-a \in M(T)$ if and only if $m+f-\frac{a}{2} \in M$, that is $f-\frac{a}{2} \in M-M=S \cup {\rm PF}(S)$, for any $\frac{a}{2} \notin S$. Thanks to Lemma \ref{almost symmetric}, it is equivalent to say that $S$ is almost symmetric. \\
{\bf (ii)} In this case $2m+2f-a \in M(T)$ if and only if $\frac{2m+2f-a-b}{2} \in E$, that is $m+f-\frac{a-b}{2}-b \in E$. This is equivalent to $f-x \in (E-M)+b$, for any $x \notin E$ and then, applying Lemma \ref{Jager} and Lemma \ref{K-E} we obtain $M-E \subseteq (E-M)+b$. \\
{\bf (iii)} The property $2e+b+2f-a \in M(T)$ is equivalent to $e+f-\frac{a}{2} \in E$, 
i.e. $f-\frac{a}{2} \in E-E$. Recalling the definition of $K$, it is equivalent to $K \subseteq E-E$. \\
{\bf (iv)} Finally, $2e+b+2f-a \in M(T)$ if and only if $e+f-\frac{a-b}{2} \in M$, i.e. $f-x \in M-E$ for any 
$x \notin E$. Using Lemma \ref{Jager} it is equivalent to say 
$K-E \subseteq M-E$ and, by Lemma \ref{K-E}, if we assume {\bf (i)} and {\bf (iii)} this fact is always true.

\end{proof}

Combining the previous theorem with Proposition \ref{duplication} and Proposition
\ref{frobenius odd}, we obtain the following corollary.

\begin{cor} \label{coro}
Let $T$ be a numerical semigroup and $S=\frac{T}{2}$. Then $T$ is almost symmetric semigroup with even type if and only if $S$ is almost symmetric and there exist an odd integer $b \in S$ and a relative ideal $E$ of $S$ such that \\\\
{\rm (1)} $T= \du$; \\
{\rm (2)} $M(S)-E \subseteq (E-M(S))+b$; \\
{\rm (3)} $K(S) \subseteq E-E$; \\
{\rm (4)} $2f(S) > 2f(E)+b$.
\end{cor}

We denote by $m(E)$ the smallest integer of $E$.

\begin{lem} \label{0}
Let $S$ be a numerical semigroup, $b \in S$ odd and $E$ a relative ideal of $S$ such that $E+E+b \subseteq S$. Then there exist $b'$ and $E'$ such that $\du=S \! \Join^{b'} \! E'$ and the smallest element of $E'$ is zero.
\end{lem}

\begin{proof}
Set $E':=E-m(E)$ and $b':=b+2m(E)$. Clearly $m(E')=0$ and $E'+E'+b'=E-m(E)+E-m(E)+b+2m(E)=E+E+b \subseteq S$; 
moreover, we have $b' \in E+E+b \subseteq S$ and, if $e \in E$, then $2e+b=2(e-m(E))+b+2m(E) \in 2E'+b'$ and vice versa; hence we have the thesis.
\end{proof}

\begin{rem}
Let $S$ be an almost symmetric numerical semigroup; we want to know which almost symmetric semigroups $T$ with even type satisfy $S=\frac{T}{2}$. Assume that the smallest element of $E$ is zero. According to Corollary 
\ref{coro}, one has $2f(S) > 2f(E)+b$, then $b<2f(S)-2f(E)\leq 2f(S)+2$. Hence we have a finite number of possibilities for $b$; moreover, if we fixed $b$, we have $-1 \leq f(E)<f(S)-\frac{b}{2}$ and then there are finitely many choices for $E$.

This fact is obvious, since $f(T)=2f(S)$ and there is a finite number of 
semigroups with fixed Frobenius number; however this remark is useful for the next example.
\end{rem}

\begin{ex}
Consider the pseudo-symmetric semigroup $S=\{0,3,5 \rightarrow \}$; 
we want to construct all almost symmetric semigroups $T$ with even type
such that $S=\frac{T}{2}$.

As in the previous remark we have $b<10$ and, 
for a fixed $b$, $-1 \leq f(E)<4-\frac{b}{2}$. 
In view of Lemma \ref{0} we are looking for only ideals containing zero and 
then $f(E)$ is different from zero. We have four possibilities:

$$
\begin{array}{ll}
b=3 \ \ \ \ \  \Longrightarrow \ \ \ \ \  f(E)=-1,1,2. \\
b=5 \ \ \ \ \  \Longrightarrow \ \ \ \ \  f(E)=-1,1. \\
b=7 \ \ \ \ \  \Longrightarrow \ \ \ \ \  f(E)=-1. \\
b=9 \ \ \ \ \  \Longrightarrow \ \ \ \ \  f(E)=-1.
\end{array}
$$

The unique ideals with Frobenius number $-1$ and $1$ are, respectively, $E_1=\mathbb{N}$ and $E_2=\{0,2 \rightarrow \}$, while there are two ideals with Frobenius number 2, 
$E_3=\{0,3 \rightarrow \}$ and $E_4=\{0,1,3 \rightarrow\}$. 
Note that, if $b=3$, $E_1$ and $E_4$ are not acceptable, because, in this case, $E+E+b \nsubseteq S$. 
It is straightforward to check that $E_i-E_i=E_i$ for $i=1,2,3$ and then $K=\{0,2,3,5 \rightarrow \}$ is contained in $E_i-E_i$ for $i=1,2$ but not for $i=3$. Finally we have
$$
\begin{array}{ll}
M-E_1=\{5 \rightarrow\},\ \ \ \ \ &E_1-M=\{-3 \rightarrow\}, \\
M-E_2=\{3,5 \rightarrow \}, &E_2-M=\{-3,-1 \rightarrow \}
\end{array}
$$
and then we obtain:
$$
\begin{array}{ll}
b=3\ \ \ \ \ \ \ \ \ \ &M-E_2 \subseteq (E_2-M)+b, \\
b=5 &M-E_1 \subseteq (E_1-M)+b, \\
&M-E_2 \nsubseteq (E_2-M)+b, \\
b=7 &M-E_1 \subseteq (E_1-M)+b, \\
b=9 &M-E_1 \nsubseteq (E_1-M)+b.
\end{array}
$$
Hence we have three possibilities and they give the numerical semigroups
$$
\begin{array}{ll}
S \! \Join^3 \! E_2=\{0,3,6,7,9 \rightarrow\}, \\
S \! \Join^5 \! E_1=\{0,5,6,7,9 \rightarrow\}, \\
S \! \Join^7 \! E_1=\{0,6,7,9 \rightarrow\}.
\end{array}
$$
Note that the first two are pseudo-symmetric, while the last one is almost symmetric with type four.
\end{ex}

\begin{lem}
For any almost symmetric numerical semigroup $S$ different from $\mathbb{N}$, 
there exists at least one relative ideal $E$ and one odd integer $b \in S$ 
such that $\du$ is almost symmetric with even type.
\end{lem}

\begin{proof}
We set $E:=\mathbb{N}$ and $b:=f+1$ if it is odd, or otherwise $b:=f+2$.

First of all note that if $e,e' \in E$, one has $e+e'+b >f$, 
then $E+E+b \subseteq S$; moreover $2f(E)+b=-2+b \leq -2 +f+2 < 2f$. 
Then, by Theorem \ref{main even}, we have to prove that $K \subseteq E-E$ 
and $M-E \subseteq (E-M)+b=(E+b)-M$. It is straightforward to check that 
$E-E=\mathbb{N}$ and $M-E=\{f+1 \rightarrow \}$, then clearly 
$K \subseteq E-E$ and if $m \in M$ and $x \in M-E$ one has $m+x \geq f+2$. 
Hence $m+x \in \{f+2, \rightarrow \} \subseteq \{b \rightarrow \}=E+b$ and 
consequently $M-E \subseteq ((E+b)-M)$ as required.
\end{proof}

We have already seen that one half of an almost symmetric numerical semigroup with even type 
is almost symmetric and the previous lemma proves the converse: 
every almost symmetric numerical semigroup is one half of some almost symmetric semigroup with even type. 
Then we can state the last corollary:

\begin{cor} \label{final}
A numerical semigroup different from $\mathbb{N}$ is almost symmetric if and only if it is one half of an almost symmetric numerical semigroup with even type.
\end{cor}

Notice that, if $\mathbb{N}$ is one half of a semigroup $T$, 
then $T$ contains all even positive integer. Hence $f(T)$ 
is odd and it is easy to see that $T$ is symmetric.

\medskip

\noindent \textbf{Acknowledgments.} The author would like to thank Marco D'Anna for his help and support 
during the drafting of the paper and Pedro Garc\'ia-S\'anchez for his useful suggestions.

\end{document}